\documentclass[12pt,a4paper]{amsart}
\usepackage{amsopn,amssymb,mathrsfs,color}
\newtheorem{theorem}{Theorem}
\newtheorem{lemma}[theorem]{Lemma}
\newtheorem{proposition}[theorem]{Proposition}
\newtheorem{corollary}[theorem]{Corollary}

\theoremstyle{definition}
\newtheorem{definition}[theorem]{Definition}

\theoremstyle{remark}
\newtheorem{example}[theorem]{Example}
\newtheorem{remark}[theorem]{Remark}

\DeclareMathOperator{\card}{\#}
\DeclareMathOperator{\ext}{ext}
\DeclareMathOperator{\diam}{diam}
\DeclareMathOperator{\s}{span}
\DeclareMathOperator{\supp}{supp}

\newcommand{\cl}[1]{\ensuremath{\overline{{#1}}}}
\newcommand{\dset}[2]{\left\{{#1} \;:\;\, {#2}\right\}}
\newcommand{\eps}{\ensuremath{\varepsilon}}
\newcommand{\Ksig}{\ensuremath{K_\sigma}}
\newcommand{\n}[1]{\ensuremath{\left\|{#1}\right\|}}
\newcommand{\N}{\mathbb{N}}
\newcommand{\ndot}{\ensuremath{\left\|\cdot\right\|}}
\newcommand{\R}{\mathbb{R}}
\newcommand{\set}[2]{\{{#1} \;:\;\, {#2}\}}
\newcommand{\st}{($*$)}
\newcommand{\tri}{{\displaystyle |\kern-.9pt|\kern-.9pt|}}
\newcommand{\tn}[1]{\ensuremath{\tri{#1}\tri}}
\newcommand{\tndot}{\ensuremath{\tri\cdot\tri}}

\allowdisplaybreaks

\begin{document}
\title{Polyhedrality in pieces}
\begin{abstract}
The aim of this paper is to present two tools, Theorems \ref{renorm-pieces}
and \ref{boundary}, that make the task of finding equivalent polyhedral
norms on certain Banach spaces easier and more transparent. The hypotheses
of both tools are based on countable decompositions, either of the unit
sphere $S_X$ or of certain subsets of the dual ball $B_{X^*}$ of a given
Banach space $X$. The sufficient conditions of Theorem \ref{renorm-pieces}
are shown to be necessary in the separable case. Using Theorem \ref{boundary},
we can unify two known results regarding the polyhedral renorming of certain
$C(K)$ spaces, and spaces having an (uncountable) unconditional basis. New
examples of spaces having equivalent polyhedral norms are given in the final section.
\end{abstract}

\author{V.\ P.\ Fonf, A.\ J.\ Pallar\'es, R.\ J.\ Smith, S. Troyanski\\ \today}

\address{Department of Mathematics, Ben Gurion University of the Negev, Beer-Sheva, Israel}\email{fonf@math.bgu.ac.il}
\address{Departamento de Matem\'aticas, Universidad de Murcia, Campus de
Espinardo, 30100 Murcia (Spain)} \email{apall@um.es}
\address{School of Mathematical Sciences, University College Dublin, Belfield, Dublin 4, Ireland} \email{richard.smith@ucd.ie}
\address{Departamento de Matem\'aticas, Universidad de
Murcia, Campus de Espinardo. 30100 Murcia (Spain).} \email{stroya@um.es}

\thanks{The first named author is supported by Israel Science
Foundation, Grant 209/09. The second named author is supported by MCI
MTM2011-25377. The third and fourth named authors are supported
financially by Science Foundation Ireland under Grant Number `SFI
11/RFP.1/MTH/3112'. The fourth named author is also supported by MCI
MTM2011-22457 and the project of the Institute of Mathematics and Informatics,
Bulgarian Academy of Science.}
\date{\today}
\maketitle

\section{Introduction}

Different notions of polyhedrality in infinite-dimensional spaces
were considered in \cite{FV}, as well as the relations between them.
In this paper, we consider the original notion of polyhedrality given
by Klee \cite{Klee}:\ a Banach space $X$ is said to be {\em polyhedral}
when the unit balls of its finite dimensional subspaces are polytopes.

We are interested in finding conditions that allow us to replace the
norm on a given Banach space with an equivalent polyhedral norm. When
a Banach space can be renormed in this way, it is called {\em isomorphically
polyhedral}. Let us recall the main tool used in this endeavour.

\begin{definition}[The {\st} property]
Let $X$ be a Banach space. We say that a set $F\subset X^*$ has
property {\st}  if, for every $w^*$-limit point $g$ of $F$ (i.e.\
any $w^*$-neighbourhood of $g$ contains infinitely many points of
$F$), we have $g(x) < 1$ whenever $\sup\set{f(x)}{f\in F}=1$.
\end{definition}

It is known (see e.g.\ \cite[Proposition 6.11]{FLP}) that $X$ is
polyhedral whenever there is a 1-norming subset $B\subset
B_{X^*}$ of the dual unit ball having {\st}. Most of the notions
considered in \cite{FV} imply that there exists one of these 1-norming
sets.

Recall that a subset $B\subset B_{X^*}$ is called a {\em boundary}
if, for every $x \in S_X$, there exists $f_x\in B$ such that $f_x(x) = 1$.
For any space, the unit sphere $S_{X^*}$ and the extreme points $\ext(B_{X^*})$ of
the unit ball are boundaries, by the Hahn-Banach and Krein-Milman Theorems,
respectively. If $B$ is a 1-norming subset of the dual unit ball having {\st},
then it is automatically a boundary. The set of boundaries of a space is
highly sensitive to changes to the norm. It is known that if $X$ is a separable
Banach space, then the following three statements are equivalent:\ $X$ is isomorphically polyhedral,
$X$ admits an equivalent norm that supports a {\em countable} boundary, and
$X$ admits an equivalent norm that supports a countable boundary having {\st} \cite{fonf2}.

However, finding renormings that support boundaries having {\st}, countable
or otherwise, can be quite difficult, even for concrete separable spaces.
In this paper we introduce two tools, Theorems \ref{renorm-pieces} and \ref{boundary},
that can be used, often in conjunction, to make the task of finding polyhedral
renormings easier and more transparent. The
hypotheses of both tools are based on countable decompositions,
either of the unit sphere $S_X$ or of certain subsets of the dual ball
$B_{X^*}$ (or both, in the case of Theorem \ref{renorm-pieces}). We believe
that these hypotheses are easier to verify in many concrete cases. We mention also
that the hypotheses of Theorem \ref{renorm-pieces} are also necessary
in the case of separable Banach spaces (see Proposition \ref{nec-condition}).

%
We introduce Theorems \ref{renorm-pieces} and \ref{boundary} in this section,
together with some of their corollaries. The proofs of these theorems reside in
subsequent sections. Some applications and examples are given
in the final section.

Before presenting Theorem \ref{renorm-pieces}, we need to generalize the notion of boundary.

\begin{definition}[Relative boundaries]
Let $X$ be a Banach space.  We shall say that a set $F\subset X^{*}$ is a
{\em relative boundary} if, 
whenever $x\in X$ satisfies $\sup\set{f(x)}{f \in F} = 1$,
then there exists $f_x\in F$ such that $f_x(x) = 1$.
\end{definition}

Relative boundaries are sometimes called {\em James boundaries} in the
literature. A different generalization of boundary for pieces of a space
can be found in \cite{FL}.

It is easy to prove that each bounded set $F\subset X^*$ having property
{\st} is a relative boundary. For countable sets, the converse holds up to a
perturbation. The following statement is the antecedent of the main results of this paper.

\begin{lemma}[{cf \cite[Theorem 3]{fonf2}}]\label{lemma4}
Let $\set{f_n}{n\in \N} \subset X^*$ be a countable relative boundary, and
let $a_n > 1$, $n\in \N$, form a decreasing sequence that converges to 1. Then the set
$\set{a_n f_n}{n \in \N}$ has {\st}.
\end{lemma}

Lemma \ref{lemma4} is implied by Proposition \ref{prop4ext}, so we
omit its proof. The following type of question arises naturally. Imagine that
$X$ admits a sequence of countable relative boundaries $\set{F_n}{n\in \N}$,
such that $F= \bigcup_{n=1}^\infty F_n$ is 1-norming. Does $X$ admit
an equivalent norm having a countable boundary? The answer is no. For
any separable Banach space $X$, there is a 1-norming sequence $f_n\in S_{X^*}$,
$n \in \N$. Clearly, the singletons $F_n=\{f_n\}$ form a sequence of
relative boundaries having {\st}, yet $X$ need not be
isomorphically polyhedral, or admit an equivalent norm having a countable
boundary. For this reason, if we wish to glue or piece
together a sequence of relative boundaries to obtain for instance a countable boundary
having {\st} (with respect to a new norm), we require some additional conditions.

Our result in this sense is the following.

\begin{theorem}\label{renorm-pieces}
Let $X$ be a Banach space and suppose that we have sets $S_n \subset S_X$
and an increasing sequence $H_n \subset B_{X^*}$ of relative boundaries,
such that $S_X = \bigcup_{n=1}^\infty S_n$ and the numbers
\[
b_n \;=\; \inf\set{\sup\set{h(x)}{h \in  H_n}}{x \in S_n},
\]
are strictly positive and converge to 1. Then, for a suitable sequence
$(a_n)_{n=1}^\infty$, the set $F = \bigcup_{n=1}^\infty a_n(H_n\setminus H_{n-1})$
is a boundary of an equivalent norm $\tndot$. Moreover, if each $H_n$
has {\st}, then $F$ has {\st} and $\tndot$ is polyhedral.
\end{theorem}

We postpone the proof of the theorem to Section \ref{pieces-proof}.
The condition that $b_n > 0$ for all $n$ is essential. For instance, take
$X=\ell_2$ with the canonical basis $\{e_n\}_{n=1}^\infty$. Set $H_n = \{\pm
e_1^*,\dots, \pm e_n^* \}$, $S_1=S_{\ell_2}$ ($b_1=0$) and $S_n =
\{e_n\}$ for $n\geq 2$ ($b_n = 1$). Each $H_n$ has {\st},
but $\ell_2$ has no equivalent polyhedral norms. The condition $b_n
\to 1$ is also essential. Given any 1-norming set $\{f_n\}_{n=1}^\infty \subset
S_{X^*}$ of a separable Banach space $X$, we can define
$H_n=\{f_1,\dots,f_n\}$ and, given $0 < \alpha <1$, let $b_n = \alpha$
and $S_n = \set{x\in S_X}{f_n(x)\geq\alpha}$.

For separable Banach spaces, Theorem \ref{renorm-pieces} has
some reasonably direct consequences that enable us to obtain some known
results, as well as yielding new examples of isomorphically polyhedral
spaces (see Example \ref{new-example}). In the case of spaces with monotone
bases, it is possible to combine Theorems \ref{renorm-pieces} and
\ref{boundary} to obtain a sufficient condition that does not refer to
the dual space (see Corollary \ref{uncbasissplitting}).

In separable spaces isomorphic polyhedrality is equivalent to the
existence of equivalent norms supporting countable boundaries.
It is natural to ask what linear topological conditions could be imposed
on boundaries to obtain a corresponding result in the general non-separable case.
The answer to this question is less clear and, to date, there is no known analogue.
Some partial results are known. We remark that the existence of a
norm-discrete and $w^*$-separable boundary is insufficient
to guarantee the existence of polyhedral renormings:\ consider for
instance the space $C(K)$, where $K$ is Kunen's compact $S$-space
(see \cite{JM} and \cite[p.\ 450]{FPST}). On the other hand, if for some
$r \in [0,1)$ the Banach space $X$ satisfies $\left(\ext{B_{X^*}}\right)'
\subset rB_{X^*}$, then $X$ is polyhedral (see e.g.\ \cite{FV}).

We go some way to bridging the gap in our knowledge by introducing the main concept
required for Theorem \ref{boundary}. Our aim is to introduce a fairly
widely applicable method of manufacturing boundaries having {\st}
for general Banach spaces. As alluded to in the abstract, we can use this method to unify \cite[Theorems 11 and
24]{FPST}. In order to do this, we introduce a strengthening of the
notion of sets having {\em small local diameter} \cite[p.\
162]{jnr:92}.

\begin{definition}\label{banach-lrc}
Let $X$ be a Banach space. We say that $E \subset X^*$ is {\em $w^*$-locally
relatively norm-compact} (or $w^*$-LRC for short) if, given
$y \in E$, there exists a $w^*$-open set $U \subset X^*$, such that
$y \in U$ and $\cl{E \cap U}^{\ndot}$ is norm-compact.

In addition, $E \subset X^*$ is called {\em $\sigma$-$w^*$-locally relatively
norm-compact} ($\sigma$-$w^*$-LRC for short) if it is the union of countably
many $w^*$-LRC sets.
\end{definition}

Let us recall that an M-basis $\{x_i,x_i^*\}_{i \in I}$ of a Banach space
is called {\em strong} if $x$ is in the norm-closed span of
$\{ x_i^*(x) x_i : i\in I\}$, and that Schauder bases and uncountable
unconditional bases are instances of strong M-bases.

\begin{example}\label{weakstarlrcex}$\;$
\begin{enumerate}
\item Any norm compact or $w^*$-discrete subset of a dual space is $w^*$-LRC.
\item Let $\{x_i,x_i^*\}_{i \in I}$ be a strong M-basis of a Banach space
$X$. Given $f \in X^*$, define
\[
\supp(f) \;=\; \set{i \in I}{f(x_i) \neq 0}.
\]
Since the basis is strong, we know that $f \in
\cl{\s}^{w^*}(x_i^*)_{i \in \supp f}$. Suppose that $E \subset X^*$
has the property that $\card(\supp(f))=\card(\supp(g)) < \infty$
whenever $f,\,g \in E$. Then $E$ is $w^*$-LRC. Indeed, if $f \in E$
then define the $w^*$-open set
\[
U = \set{g \in X^*}{{\textstyle 0 < |g(x_i)| < |f(x_i)|+1}
\text{ for all }i \in \supp f}.
\]
It is clear that if $g \in E \cap U$, then $\supp(g) = \supp(f)$.
Thus $E \cap U$ is a norm bounded subset of a finite-dimensional space.
\end{enumerate}
\end{example}

Evidently, any countable set in the dual unit ball is $\sigma$-$w^*$-LRC
and $w^*$-$\sigma$-compact. The next result extends to the general case
the concept of countable boundary for polyhedral renorming. Given $\eps > 0$
and norms $\ndot$ and $\tndot$ on a Banach space $X$,
we say that $\tndot$ {\em $\eps$-approximates} $\ndot$ if
\[
(1-\eps)\n{x} \leq \tn{x} \leq (1+\eps)\n{x},
\]
whenever $x \in X$.

\begin{theorem}\label{boundary}
Let $X$ be a Banach space and suppose that $E$ is a
$\sigma$-$w^*$-LRC and $w^*$-$\sigma$-compact subset of $X^*$ that
contains a boundary of $X$. Then, given $\eps>0$, there exists a new
norm $\tndot$ on $X$ that $\eps$-approximates the original norm and
admits a boundary having {\st}. In particular, $\tndot$ is
polyhedral.
\end{theorem}

The proof of Theorem \ref{boundary} is deferred to Section \ref{boundary-proof}.
The additional requirement that $E$ in Theorem \ref{boundary} be
$w^*$-$\sigma$-compact is imposed in order to control the
locations of its limit points. We do not know to what extent this
condition can be relaxed in general, if at all.

The next corollary makes use of both Theorem \ref{renorm-pieces}
and Theorem \ref{boundary}.

\begin{corollary}\label{boundary-pieces}
Let $X$ be a Banach space, and let $H_k \subset B_{X^*}$, $k \in
\N$, be an increasing sequence of relative boundaries, such that
each may be covered by a $w^*$-$\sigma$-LRC and
$w^*$-$\sigma$-compact set. Suppose moreover that we can write $S_X$
as a union $\bigcup_{n=1}^{\infty} S_n$, in such a way that the
sequence
\[
b_n \;=\; \inf\set{\sup\set{f(x)}{f \in H_n}}{x \in S_n}
\]
is strictly positive and converges to 1. Then $X$ admits an
equivalent (polyhedral) norm that has a boundary having {\st}.
\end{corollary}

\begin{proof}
Using Theorem \ref{renorm-pieces}, we can find suitable
numbers $a_n > 1$, together with an equivalent
norm $\tndot$, with respect to which the set $F =
\bigcup_{n=1}^\infty \pm a_n (H_n\setminus H_{n-1})$ is a boundary.
If $E_n$ is a $w^*$-$\sigma$-LRC and $w^*$-$\sigma$-compact set that
covers $H_n$, then $\bigcup_{k=0}^\infty \pm a_n E_n$ is a
$w^*$-$\sigma$-LRC and $w^*$-$\sigma$-compact set that covers $F$.
Now apply Theorem \ref{boundary}.
\end{proof}

\section{The proof of Theorem \ref{renorm-pieces}}\label{pieces-proof}

We begin by stating a way of gluing together relative boundaries having
{\st}, along the lines of Lemma \ref{lemma4}.

\begin{proposition}\label{prop4ext}
Let $H_n \subset B_{X^*}$ be an increasing sequence of subsets and
$(a_n)_{n=1}^\infty$ a decreasing sequence converging to 1. Set
$H = \bigcup_{n=1}^\infty H_n$ and $F = \pm\bigcup_{n=1}^\infty a_n
(H_n\setminus H_{n-1})$, where $H_0$ is empty. If each $H_n$ has {\st}, and
\[
\sup\set{|h(x)|}{h \in H} < \sup\set{f(x)}{f \in F}
\]
whenever $x$ is non-zero, then $F$ has {\st}.
\end{proposition}

\begin{proof}
Let $g$ be a $w^*$-limit point of $F$ and suppose that
$\sup\set{f(x)}{f \in F}=1$. First, imagine that $g$ is not a
$w^*$-limit point of $\pm a_n (H_n\setminus H_{n-1})$ for any $n$.
Then it is possible to find sequences $n_k \to \infty$ and
$h_k \in H_{n_k}\setminus H_{{n_k}-1}$,
satisfying
\begin{align*}
|g(x)| \;=\; \lim_k |a_{n_k}h_k(x)| \;=\; \lim_k |h_k(x)| &\;\leq\; \sup\set{|h(x)|}{h \in H}\\
&\;<\; \sup\set{f(x)}{f \in F} \;=\; 1.
\end{align*}
Now suppose that $g$ is a $w^*$-limit point of $a_m (H_m\setminus H_{m-1})$,
for some $m$. Assume that $g(x)$ is positive, otherwise there is
nothing to prove. Since $a_m H_m$ has
{\st}, we obtain
\begin{align*}
0 \;<\; g(x) &\;<\; \sup\set{f(x)}{f \in a_m H_m} \\
&\;\leq\; \sup\set{f(x)}{f \in a_k(H_k\setminus H_{k-1}),\, k \leq m}\quad \text{(as $(a_n)$ is decreasing)} \\
&\;\leq\; \sup\set{f(x)}{f \in F} = 1.
\end{align*}
If $g$ is a $w^*$-limit point of $-a_m (H_m\setminus H_{m-1})$, then repeat the
above using $-x$.
\end{proof}

\begin{proof}[Proof of Theorem \ref{renorm-pieces}]
Let $X$ be a Banach space and suppose that we have $S_n \subset S_X$
and increasing sets $H_n \subset B_{X^*}$, such that $S_X = \bigcup_{n=1}^\infty S_n$ and
the numbers
\[
b_n \;=\; \inf\set{\sup\set{h(x)}{h \in  H_n}}{x \in S_n},
\]
are strictly positive and converge to 1.

Define $H = \bigcup_{n=1}^\infty \pm H_n$ and $n(h)=\min\set{n \in \N}{h \in H_n}$
whenever $h \in H$. Set $c_n = \inf\set{b_m}{m \geq n}$ and $a_n =
(1+2^{-n})c_n^{-1}$. We are going to prove the following

\textbf{Claim }

The seminorms
\[
\n{x}_n \;=\; \sup\set{a_{n(h)}|h(x)|}{h \in H_n}, \quad n \in \N,
\]
and the norm
\[
\tn{x} \;=\; \sup\set{a_{n(h)}|h(x)|}{h \in H},
\]
possess the following properties:

{\renewcommand{\theenumi}{\roman{enumi}}
\begin{enumerate}
\item \label{lemma:1:1} $\n{x} < \tn{x}\leq a_1 \n{x}$ for every
$x\in X\setminus\{0\}$;
\item \label{lemma:1:2} for every $x\in X$, there exists an integer
$n_x$ satisfying $\tn{x}=\n{x}_{n_x}$;
\item \label{lemma:1:3} if each $H_n$ is a relative boundary, then
$F = \bigcup_{n=1}^{\infty}\pm a_n (H_n\setminus H_{n-1})$ is a
boundary with respect to $\tndot$;
\item if each $H_n$ has {\st}, then $F$ also has {\st} and $\tndot$
is polyhedral.
\end{enumerate}}

Proof of (i): Let $x \in X\setminus\{0\}$ and $n \in \N$ such that $x/\n{x}
\in S_n$. Then $c_n \leq \sup\set{h(x/\n{x})}{h \in H_n}$, and since
$n(h) \leq n$ whenever $h \in H_n$, we obtain
\begin{align*}
\n{x} \;\leq\; \frac{1}{c_n}\sup\set{h(x)}{h \in H_n}
&\;\leq\; \frac{1}{a_n c_n}\sup\set{a_{n(h)}|h(x)|}{h \in H_n}\\
&\;=\; \frac{1}{a_n c_n}\n{x}_n \;<\; \n{x}_n \;\leq\; \tn{x} \;\leq\; a_1\n{x}.
\end{align*}

Proof of (ii): Suppose $1=\n{x} < \tn{x}$. Take $N$ such that $\tn{x} >
a_N$. If $h \in H$ and $n(h) \geq N$, then $a_{n(h)}|h(x)| \leq a_N
< \tn{x}$. Therefore, $\tn{x}=\max_{n < N} \n{x}_n$.

Proof of (iii): Let $\tn{x}=\n{x}_n$, where $n\geq 1$ is minimal. If $h
\in H_n$ satisfies $\n{x}_n \geq a_{n(h)}|h(x)| > \n{x}_k$, $k < n$,
then $n(h)=n$. Therefore,
\begin{align*}
\n{x}_n &\;\leq\; a_n\sup\set{|h(x)|}{h \in H_n}\\
&\;\leq\; \sup\set{a_{n(h)}|h(x)|}{h \in H_n} \;=\; \n{x}_n.
\end{align*}
As $H_n$ is a relative boundary, there exists $g \in H_n$ such that
$a_n|g(x)|= a_n\sup\set{|h(x)|}{h \in H_n}=\n{x}_n$. Again, $n(g)=n$
by minimality.

Proof of (iv): Clearly from \eqref{lemma:1:1}, we have $\sup\set{|h(x)|}{h \in H} \leq \n{x}
< \tn{x}=\sup\set{f(x)}{f \in F}$, so we can apply Proposition
\ref{prop4ext} to conclude that $F$ has {\st}. According to \cite[Proposition 6.11]{FLP},
$\tndot$ is polyhedral.
\qedhere
\end{proof}

In the proof above, we can use any decreasing sequence
$(a_n)_{n=1}^\infty$ satisfying $a_n c_n > 1$ and $a_n \to 1$.

As we mentioned in the Introduction, the hypotheses of Theorem \ref{renorm-pieces}
for the existence of polyhedral renormings are also necessary in the
case of separable Banach spaces. This follows from the results regarding
the approximation of convex bodies by polytopes obtained in \cite{DFH}.

\begin{proposition}\label{nec-condition}
Let $(X,\ndot)$ be a separable Banach space that has an equivalent norm with a
countable boundary (isomorphically polyhedral). Then, for every $\eps >0$, there exists a
sequence of subsets $F_n\subset B_{X^*}$ having {\st}, such that
the sequence
\[
b_n \;=\; \inf \set{\sup\set{f(x)}{f \in F_n}}{x \in S_X},
\]
satisfies $b_n>1-\eps$ for every $n\in\N$, and $\lim_n b_n=1$.
\end{proposition}

\begin{proof}
In \cite[Theorem 1.1]{DFH}, it is proved that if $(X,\ndot)$ admits an equivalent
polyhedral norm, then for every $\eta >0$, there exists a polytope $P_\eta$ that
$\eta$-approximates the unit ball $B_X$ with respect to the norm $\ndot$, i.e.\
$B_X\subset P_\eta \subset (1+\eta)B_X$. Clearly $P_{\eta}$ has a countable boundary $F_\eta$,
$F_\eta \subset P_\eta^\circ \subset B_{X^*}$. Also, we can assume that $F_\eta$
has {\st} and, in particular, if $x\in S_X$, then $1\geq \sup\set{f(x)}{f \in F_\eta}\geq (1+\eta)^{-1}$.

If we choose a sequence $0<\eps_n<\eps<1$ that converges to 0, and we
 put $\eta_n = (1-\eps_n)^{-1}-1 >0$ in the
previous argument,  we get a sequence of sets $F_n\subset B_{X^*}$ having {\st} that
verifies the statement of the proposition,
\[
1\;\geq\; b_n \;=\; \inf\set{\sup\set{f(x)}{f \in F_n}}{x \in S_X} \;\geq\; 1-\eps_n\;>\;0. \qedhere
\]
\end{proof}

\section{The proof of Theorem \ref{boundary}}\label{boundary-proof}

The notion of $w^*$-LRC set given in the Introduction can be cast in
the more general context of sets endowed with a pair of topologies.

\begin{definition}\label{lrc}
Let $X$ be a set and let $\tau$ and $\rho$ be two Hausdorff topologies on $X$,
with $\rho$ finer than $\tau$. We say that $E \subset X$ is $\tau$-locally
relatively $\rho$-compact (or $(\tau,\rho)$-LRC for short), if given
$x \in E$, there exists a $\tau$-open set $U \subset X$, such that
$x \in U$ and $\cl{E \cap U}^{\rho}$ is $\rho$-compact.
\end{definition}

Before presenting the proof of Theorem \ref{boundary} and seeing how
this concept helps in the theory of isomorphic polyhedrality, we
explore some of its general consequences.

\begin{proposition}\label{lrctheory}
Let $X$, $\tau$ and $\rho$ be as in Definition \ref{lrc}.
\begin{enumerate}
\item If $(X,\tau)$ is a Baire space and $\cl{U}^\rho$ is not $\rho$-compact
whenever $U$ is $\tau$-open and non-empty, then $X$ is not the union of
countably many $(\tau,\rho)$-LRC subsets.
\item If $\rho$ is metrizable (with metric also denoted by $\rho$), then
any $(\tau,\rho)$-LRC set $E$ has small local diameter, i.e., given $x \in E$
and $\eps>0$, there exists a $\tau$-open subset $U \subset X$, such
that $x \in E \cap U$ and $\rho$-$\diam{E \cap U} < \eps$.
\item If $E$ is $(\tau,\rho)$-LRC then there exists a $\tau$-open set
$V$, such that $E \subset \cl{E}^\tau \cap V \subset \cl{E}^\rho$,
and $\cl{E}^\tau \cap V$ is also $(\tau,\rho)$-LRC.
\end{enumerate}
\end{proposition}

\begin{proof}$\;$
\begin{enumerate}
\item It is enough to show that if $E\subset X$ is $(\tau,\rho)$-LRC,
then $E$ is $\tau$-nowhere dense. Indeed, suppose that $U\neq\varnothing$
is $\tau$-open and $U \subset \cl{E}^\tau$.
Given $y \in E \cap U$, we can find a $\tau$-open set $V \subset U$ such that
$y \in V$ and $\cl{E \cap V}^{\rho}$ is $\rho$-compact. Since
$\rho$ is Hausdorff and finer than $\tau$, it follows that $\cl{E \cap V}^{\rho}$
is $\tau$-closed, thus
\[
V \;=\; \cl{E}^\tau \cap V \;\subset\; \cl{E \cap V}^\tau \;=\; \cl{E \cap V}^\rho,
\]
and $\cl{V}^\rho = \cl{E \cap V}^\rho$ is $\rho$-compact, contrary to assumption.
\item Let $x \in E$ and let $\cl{E \cap U}^\rho$ be $\rho$-compact, where
$U$ is $\tau$-open and $x \in U$. As above, $\cl{E \cap U}^\rho = \cl{E \cap U}^\tau$.
Since the identity map from $(\cl{E \cap U}^\rho,\rho)$ to $(\cl{E \cap U}^\tau,\tau)$
is continuous and bijective, it follows that $\rho$ agrees with $\tau$ on this set.
In particular, if $W \subset E \cap U$ is $\rho$-relatively open, then there exists
a $\tau$-open subset $V \subset X$ such that $W = E \cap V$. The result follows.
\item Define
\[
V \;=\; \bigcup\set{U \subset X}{\text{$U$ is $\tau$-open
and $\cl{E \cap U}^\rho$ is $\rho$-compact}}.
\]
Certainly, $E \subset \cl{E}^{\tau} \cap V$. Moreover,
if $\cl{E \cap U}^\rho$ is $\rho$-compact then
$\cl{E}^{\tau} \cap U \subset \cl{E \cap
U}^{\tau} =  \cl{E \cap U}^{\rho} \subset
\cl{E}^{\rho}$, so $\cl{E}^{\tau} \cap V
\subset \cl{E}^{\rho}$. The inclusions in the previous
sentence also show that $\cl{E}^{\tau} \cap V$ is
$(\tau,\rho)$-LRC. \qedhere
\end{enumerate}
\end{proof}

From Proposition \ref{lrctheory} (1), it follows immediately that if
$X$ is infinite-dimensional and $W \subset X$ has a non-empty norm
interior, then it cannot be covered by countably many
$(w,\ndot)$-LRC sets (any $(w,\ndot)$-LRC set is also
$(\ndot,\ndot)$-LRC, of course). Likewise if we consider dual Banach
spaces in the $w^*$-topology. Thus, according to Proposition
\ref{lrctheory} (2), the notion of $(\tau,\rho)$-LRC sets is
strictly stronger than that of sets having small local diameter. Indeed,
consider any Banach space $X$ with an equivalent Kadec norm
\cite[Theorem 2.1]{jnr:92}.

However, while it isn't possible to cover $X^*$ with countably many
$w^*$-LRC sets (assuming $\dim X  = \infty$), it is sometimes
possible to cover {\em boundaries} with countably many such sets.

The proof of Theorem \ref{boundary} combines two ideas. The first
(the use of a function of type $\psi$, as below)
originates in \cite[Theorem 5.1]{hh:07} and matures in \cite[Theorem
11]{FPST} and \cite[Lemma 3]{st:10}. The second idea is based on
the use of finite $\eps$-nets, which are employed in this, and
related contexts, in e.g.\ \cite[Theorem 24]{FPST}, \cite{fonf3} and
\cite{H1}.

\begin{proof}[Proof of Theorem \ref{boundary}]
Let $E = \bigcup_{n=0}^\infty E_n$, where each $E_n$ is $w^*$-LRC.
We can assume that $E$ is a boundary, and that $\cl{E_n}^{w^*}
\subset E$ for all $n \in \N$. Indeed, if necessary, consider
\[
E_n \cap K_m \cap B_{X^*}
\]
for $n,\, m \in \N$, where $E = \bigcup_{m=0}^\infty K_m$ and each
$K_m$ is $w^*$-compact.

By Proposition \ref{lrctheory} (3), there exist $w^*$-open sets $V_n$, such
that if we set $A_n = \cl{E_n}^{w^*} \cap V_n$, then $E_n
\subset A_n \subset \cl{E_n}^{\ndot}$ and $A_n$ is
$w^*$-LRC. We can see that each $A_n$ is both norm $F_\sigma$
and norm $G_\delta$, so write
\[
A_n\setminus\bigcup_{k<n} A_k \;=\; \bigcup_{m=0}^\infty H_{n,m},
\]
where each $H_{n,m}$ is norm closed and $H_{n,m}\subset H_{n,m+1}$
for all $m \in \N$. For convenience, set $H_{n,-1} = \varnothing$.
Let $\pi:\N^2 \to \N$ be a bijection and define
\[
L_{\pi(i,j)} \;=\; H_{i,j}\setminus H_{i,j-1},
\]
for $i,\, j \in \N$. It is clear that $E$ is the disjoint union of
the $L_n$, and that $\cl{L_n}^{w^*} \subset
\cl{E_p}^{w^*} \subset E$, where $n=\pi(p,q)$.

Fix $\eps>0$. Given $f \in E$, let
\[
I(f) \;=\; \set{n \in \N}{f \in \cl{L_n}^{w^*}}.
\]
and $n(f) = \min I(f)$. Define a function
$\psi:E \to (1,1+\eps)$ by
\[
\psi(f) = 1 + {\textstyle \frac{1}{2}}\eps\cdot 2^{-n(f)}\left(1 +
{\textstyle \frac{1}{4}}\sum_{i \in I(f)} 2^{-i} \right).
\]
Set $\eps_n = \frac{1}{160}\eps\cdot 4^{-n}$. Fix $n$ for a moment.
Because $\psi(L_n) \subset (1, 1+\eps)$, there is a finite partition
of $L_n$ into sets $J$, such that $\diam{\psi(J)} \leq \eps_n$. For
each such $J$, we can use Zorn's Lemma to extract a
$\eps_n$-separated subset $\Gamma$ of $J$ that is maximal. The
maximality guarantees that $\Gamma$ is a $\eps_n$-net, while the
$\eps_n$-separation implies that if $M \subset J$ is totally
bounded, then $M \cap \Gamma$ is finite. Therefore, by considering
the finite union of these $\eps_n$-separated subsets, there exists a
subset $\Gamma_n$ of $L_n$, with the property that if $f \in L_n$,
then we can find $h \in \Gamma_n$ satisfying
\begin{equation}\label{eqn1}
|\psi(f)-\psi(h)|,\, \n{f-h} \;\leq\; \eps_n.
\end{equation}
Moreover, if $M \subset L_n$ is totally bounded, then $M \cap
\Gamma_n$ is finite.

Now define $B = \bigcup_{n=0}^\infty \Gamma_n$, $D =
\set{\psi(f)f}{f \in B}$ and
\[
\tn{x} \;=\; \sup\set{\psi(f)f(x)}{f \in B}.
\]

\textbf{Claim 1}

If $x \neq 0$, then $\n{x} < \tn{x} \leq (1 + \eps)\n{x}$.

Obviously $\tn{x} \leq (1 + \eps)\n{x}$, because $\psi \leq 1
+ \eps$ and $B \subset E$. Now let $\n{x} = 1$. Since $E$ is
assumed to be a boundary, take $k \in E$ satisfying $k(x) = 1$.
There is a unique $n$ such that $k \in L_n$. It follows that we can
find $f \in \Gamma_n$ such that $\n{k-f} \leq \eps_n$. Now
\begin{align*}
1 = k(x) &= f(x) + (k-f)(x)\\
&= \psi(f)f(x) + (1-\psi(f))f(x) + (k-f)(x)\\
&\leq \tn{x} + (1-\psi(f))f(x) + (k-f)(x),
\end{align*}
so we are done if we can show that $(\psi(f)-1)f(x) + (f-k)(x) > 0$.
Now $\psi(f) \geq 1 + \frac{1}{2}\eps\cdot 2^{-n(f)}$ and $n\geq
n(f)$. Certainly,
\[
f(x) \;\geq\; 1 - \n{k-f}\n{x} \geq {\textstyle \frac{1}{2}},
\]
therefore
\[
(\psi(f)-1)f(x) + (f-k)(x) \;\geq\; {\textstyle \frac{1}{4}}\eps\cdot
2^{-n(f)} - \eps_n \;\geq\; {\textstyle \frac{1}{4}}\eps\cdot 2^{-n} -
{\textstyle \frac{1}{160}}\eps\cdot 4^{-n} \;>\; 0,
\]
which completes the proof of the claim.

Obviously, $D$ is $1$-norming with respect to $\tndot$. It
remains to show that $D$ has {\st}. To this end, fix $g \in
D^\prime$ and $x\in X$ satisfying $\tn{x} = 1$. We must prove
that $g(x) < 1$. Let $(f_\lambda) \subset B$ be a net satisfying
$\psi(f_\lambda)f_\lambda \stackrel{w^*}{\to} g$ and
$\psi(f_\lambda)f_\lambda \neq g$. By taking a subnet if necessary,
again denoted $(f_\lambda)$, we can assume that $\psi(f_\lambda) \to
\alpha$, for some $\alpha \geq 1$. It follows that $f_\lambda
\stackrel{w^*}{\to} g/\alpha$. If $\alpha=1$ then
\[
g(x) \;\leq\; \sup\set{f(x)}{f \in B} \;\leq\; \n{x} \;<\;
\tn{x} \;=\; 1,
\]
so it remains to tackle the case when $\alpha > 1$.

\textbf{Claim 2}

If $\alpha > 1$ then there exists $n \in \N$ such that
$\tn{g} \leq 1 - \eps_n < 1$. In particular, this will yield
$g(x)<1$.

We prove this statement using a number of subclaims. Fix $N$ large
enough so that $1+\eps\cdot 2^{-N} < \frac{1}{2}(1+\alpha)$. If $f
\in E$, then $\psi(f) \leq 1 + \eps\cdot 2^{-n(f)}$, thus we must
have $n(f_\lambda) < N$ for large enough $\lambda$. It follows that
$f_\lambda \in \bigcup_{k<N} \cl{L_k}^{w^*}$ for large enough
$\lambda$. By $w^*$-compactness, $g/\alpha \in \bigcup_{k<N}
\cl{L_k}^{w^*} \subset E$. Hereafter, fix $f = g/\alpha$.

By extracting another subnet, we can assume that $f_\lambda \neq f$.
Indeed, note that given any $\lambda_0$, there exists $\lambda \geq
\lambda_0$ satisfying $f_\lambda \neq f$. Otherwise, we would have
$\psi(f)f =  \psi(f_\lambda)f_\lambda \stackrel{w^*}{\to} g$ for
large enough $\lambda$, giving $\psi(f_\lambda)f_\lambda = g$,
contrary to assumption.

From now on, fix the unique $n$ such that $f \in L_n$, and let
\[
J \;=\; I(f) \cup \set{m \in \N}{m \geq n+2}.
\]
Clearly, $n \in I(f)$. Fix $(p,q) \in \N^2$ such that $n=\pi(p,q)$.
We have $L_n \subset A_p$. Since $A_p$ is $w^*$-LRC, there exists a
$w^*$-open set $U$ containing $f$, such that $A_p \cap U$ is
relatively norm compact.

From the discussion above, $\Gamma_{\pi(p,k)} \cap U$ is finite for
all $k \in \N$, since $\Gamma_{\pi(p,k)} \subset A_p$. This
observation, along with the fact that $\N\setminus J$ is finite,
means that the set
\[
V \;=\; U\setminus\left(\bigcup_{i \in \N\setminus J} \cl{L_i}^{w^*}
\cup \left(\bigcup_{k=0}^q \Gamma_{\pi(p,k)}\setminus\{f\}
\right)\right)
\]
is $w^*$-open. Moreover, because $f \notin \bigcup_{i \in
\N\setminus J} \cl{L_i}^{w^*}$, we have $f \in V$. We
will assume, from now on, that $f_\lambda \in V$.

\textbf{Claim 2a}

$n \notin I(f_\lambda)$.

If $n \in I(f_\lambda)$, then
\[
f_\lambda \in \cl{L_n}^{w^*} \cap V
\;\subset\; \cl{L_n \cap V}^{w^*} \;=\;
\cl{L_n \cap V}^{\ndot} \;\subset\; \cl{L_n}^{\ndot} \;\subset\; H_{p,q}.
\]
It follows that $f_\lambda \in H_{p,k}\setminus H_{p,k-1} =
L_{\pi(p,k)}$ for some $k \leq q$. On the other hand, $f_\lambda \in
B$, so $f_\lambda \in L_{\pi(p,k)} \cap B = \Gamma_{\pi(p,k)}$.
However, this cannot be the case, since $f_\lambda \in
V\setminus\{f\}$.

\textbf{Claim 2b}

$I(f_\lambda) \subset J$.

Let $i \in I(f_\lambda)$. If $i \notin J$, then $f_\lambda \in
\bigcup_{j \in \N\setminus J} \cl{L_j}^{w^*}$, but this runs
contrary to the fact that $f_\lambda \in V$.

\textbf{Claim 2c}

$\psi(f) - \psi(f_\lambda) \geq {\textstyle \frac{1}{16}}\eps\cdot
4^{-n} = 10\eps_n$.

First, note that $n(f_\lambda) \geq n(f)$, by Claim 2b and the fact
that $n(f) = \min I(f) = \min J$. There are two cases to consider.
If $n(f_\lambda) > n(f)$, then
\begin{align*}
\psi(f) - \psi(f_\lambda) &\geq 1 + {\textstyle
\frac{1}{2}}\eps\cdot 2^{-n(f)} - (1 + {\textstyle
\frac{3}{4}}\eps\cdot 2^{-n(f_\lambda)}) \\
&\geq {\textstyle \frac{1}{8}}\eps\cdot 2^{-n(f)} \;\geq\;
{\textstyle \frac{1}{8}}\eps\cdot 2^{-n},
\end{align*}
as $n(f)\leq n$.

Instead, if $n(f_\lambda)=n(f)$, then
\begin{align*}
\psi(f) - \psi(f_\lambda) &\geq {\textstyle \frac{1}{8}}\eps\cdot
2^{-n(f)}\left(\sum_{i \in I(f)} 2^{-i} - \sum_{i \in I(f_\lambda)}
2^{-i} \right)\\
&= {\textstyle \frac{1}{8}}\eps\cdot 2^{-n(f)}\left(\sum_{i \in
I(f)\setminus I(f_\lambda)} 2^{-i} - \sum_{i \in
I(f_\lambda)\setminus I(f)} 2^{-i} \right)\\
&\geq {\textstyle \frac{1}{8}}\eps\cdot 2^{-n(f)}\left(2^{-n} -
\sum_{i
\in J\setminus I(f)} 2^{-i} \right) \qquad\text{(Claims 2a and 2b)}\\
&\geq {\textstyle \frac{1}{8}}\eps\cdot 2^{-n(f)}\cdot 2^{-n-1}
\geq {\textstyle \frac{1}{16}}\eps\cdot 4^{-n}.
\end{align*}
Now we are in a position to finish the proof of Claim 2. By
(\ref{eqn1}), choose $h \in \Gamma_n$ such that
\[
|\psi(f)-\psi(h)|,\, \n{f-h} \;\leq\; \eps_n.
\]

We have $g = \alpha f$, so
\[
g(x) \;=\; \alpha f(x) \leq \alpha h(x) + \alpha \eps_n \n{x}
\;\leq\; \alpha h(x) + 2\eps_n,
\]
since $\alpha \leq 1 + \eps < 2$ and $\n{x} < \tn{x} = 1$.

If $h(x) \leq \frac{1}{3}$, then $g(x) \leq 2(\frac{1}{3}+
\frac{1}{160}) = \frac{163}{240} \leq 1 - \eps_n$. Instead, suppose
that $h(x) > \frac{1}{3}$. From Claim 2c, we have $\psi(f) - \alpha
\geq 10\eps_n$, so $\psi(h) \geq 9\eps_n$. It follows that
\begin{align*}
g(x) &\leq \psi(h)h(x) + (\alpha - \psi(h))h(x) + 2\eps_n\\
&\leq 1 + (\alpha - \psi(h))h(x) + 2\eps_n \qquad\text{($\tn{x}
= 1$ and $h \in B$)}\\
&\leq 1 - 3\eps_n  + 2\eps_n\\
&= 1 - \eps_n.
\end{align*}
As $n$ was obtained independently of $x$, we have
$\tn{g} \leq 1 - \eps_n$.
\end{proof}

\section{Applications and Examples}\label{non-sep}

First of all, we get the following result proved in \cite{fonf3}.

\begin{corollary}[\cite{fonf3}]\label{sigma-compact-boundary}
If $(X,\ndot)$ admits a boundary contained in a norm $\sigma$-compact
subset of $B_{X^*}$, then given $\eps>0$, $X$ admits a
(polyhedral) norm $\tndot$ on $X$ that $\eps$-approximates the
original norm and admits a boundary having {\st}.
\end{corollary}

\begin{proof}
Theorem \ref{boundary} and Example \ref{weakstarlrcex} (1).
\end{proof}

Of course, the corollary above includes the case when the boundary of
$X$ is countable \cite{fonf2}.

Next, we move on to spaces with bases. In Example \ref{weakstarlrcex}, we
showed how $w^*$-LRC sets can manifest in the duals of spaces having strong M-bases.
Using sets of this kind, together with Corollary \ref{boundary-pieces}, we obtain
Corollary \ref{uncbasissplitting} which contains, as a particular case, the
polyhedral renorming of spaces having unconditional basis given in \cite[Theorem 24]{FPST}.

\begin{corollary}\label{m-basissplitting}
Let $X$ have a strong M-basis $\{x_i,x_i^*\}_{i \in I}$, and suppose
that we can write $S_X = \bigcup_{n=1}^\infty S_n$, in such a way
that the
\[
b_n \;=\; \inf\set{\sup\set{f(x)}{f \in B_{X^*},\, \card(\supp (f))
\leq n}}{x \in S_n}
\]
are strictly positive and converge to $1$. Then $X$ admits an
equivalent (polyhedral) norm that supports a boundary having {\st}.
If $I=\N$ and
\[
b_n \;=\; \inf\set{\sup\set{f(x)}{f \in B_{X^*},\, \max(\supp(f))
\leq n}}{x \in S_n}
\]
behave likewise, then we obtain a norm as above, supporting a countable
boundary having {\st}.
\end{corollary}

\begin{proof}
From Example \ref{weakstarlrcex} (2), we know that
\[
H_n \;=\; \set{f \in B_{X^*}}{\card(\supp(f)) \leq n}
\]
is a finite union of $w^*$-LRC sets. Moreover, each $H_n$ is a
relative boundary because it is $w^*$-compact. Now
apply Corollary \ref{boundary-pieces}. In the second case, consider
the norm-compact sets $\set{f \in B_{X^*}}{\max(\supp(f)) \leq n}$.
Observe that in this totally bounded situation, the sets $\Gamma_n$ obtained
in the proof of Theorem \ref{boundary} are necessarily finite,
and thus the boundary constructed will be countable.
\end{proof}

We note that it is possible to obtain the second statement of the result
above without resorting to the heavy machinery of Theorem \ref{boundary}.

Given a space $X$ with an M-basis $\{x_i,x_i^*\}_{i \in I}$, and a finite
subset $\sigma \subseteq I$, let $P_\sigma$ denote the projection given by
$P_\sigma(x) = \sum_{i \in \sigma} x_i^*(x)x_i$. If $I=\N$, let $P_n = P_{\{1,\dots, n\}}$.

\begin{corollary}[cf {\cite[Theorem 24]{FPST}}]\label{uncbasissplitting}
Let $X$ have a monotone unconditional basis $\{e_i\}_{i \in I}$, with
associated projections $P_\sigma$, $\sigma \subset I$, and suppose
that we can write $S_X = \bigcup_{n=1}^\infty S_n$ in such a way that
the numbers
\[
b_n \;=\; \inf\set{\sup\set{\n{P_\sigma(x)}}{\sigma \subset I,\, \card(\sigma) = n}}{x \in S_n}
\]
are strictly positive and converge to $1$. Then $X$ admits an
equivalent (polyhedral) norm supporting a boundary having {\st}. If
$\{e_n\}_{n=1}^\infty$ is a monotone Schauder basis and
$b_n = \inf\set{\n{P_n(x)}}{x \in S_n}$ behave likewise, then we obtain
a norm as above, supporting a countable boundary having {\st}.
\end{corollary}

\begin{proof}
Given $x \in X$ and $\sigma \subset I$, take $f\in B_{X^*}$ such
that $\n{P_\sigma(x)} = f(P_\sigma(x)) = (P_\sigma^*(f))(x)$. Clearly
$\supp P_\sigma^*(f) \subset \sigma$, and $P_\sigma^*(f) \in B_{X^*}$
as the basis is monotone. Now apply Corollary \ref{m-basissplitting}. The second
statement is dealt with similarly.
\end{proof}

In \cite[Theorem 3]{leung} it is proved, under the assumption that
$X$ is a Banach space with a shrinking basis $\{e_n\}_n$, that (a)
$X$ is isomorphically polyhedral if, and only if, (b) there exists an
equivalent norm, $\ndot_L$, such that $\{e_n\}_n$ is monotone, and 
for every $x\in X$ there exists $n_x \in \N$ satisfying $\n{x}_L=\n{P_{n_x}(x)}_L$.
In this result, the shrinking basis is used only in the proof that
(a) implies (b). Note that, if $X$ is a Banach space with a monotone
Schauder basis satisfying the hypothesis of Corollary \ref{uncbasissplitting},
namely $b_n = \inf\set{\n{P_n(x)}}{x \in S_n}$ is strictly positive and
converges to 1, then $X$ has an equivalent norm as in (b). This norm
can be obtained using Theorem \ref{renorm-pieces} (see statements (i)
to (iii) of the claim in its proof), applied to the sequence of norm
compact sets $H_n=P_n^*(B_{X^*})$. In this case we do
not need to check that the basis is shrinking to obtain (b).

With the help of Corollary \ref{uncbasissplitting}, we can obtain new
examples of isomorphically polyhedral spaces.

\begin{example}\label{new-example}
Let $\mathcal{A}=(A_n)_{n=0}^{\infty}$ be a sequence of finite
subsets of natural numbers $A_n \subset \N$, such that $\N =
\bigcup_{n=0}^{\infty}A_n$. Let $\mathbf{p} = \left(p_n\right)_{n=0}^{\infty}$
be an unbounded increasing sequence of real numbers $p_n \geq 1$
(necessarily tending to $\infty$). 

For each sequence of real numbers $x:\N \to \R$, set
\[
\Phi(x) \;=\; \sup\dset{\sum_{n=0}^{\infty} \sum_{k\in B_n}
|x(k)|^{p_n}}{B_n\subset A_n \text{ and } B_n \text{ are pairwise disjoint}}.
\]

Clearly, $\Phi$ is a convex function taking values in $[0,+\infty]$, and $\Phi(0)=0$.
By $\ell_{\mathcal{A},\mathbf{p}}$ we denote the space of all
sequences $x$ for which there is some $\lambda >0$ satisfying
$\Phi( x / \lambda) < +\infty$. As in the case of the Orlicz
sequence spaces, the function
\[
\n{x} \;=\; \inf\dset{\lambda>0}{\Phi( x / \lambda) \leq 1},
\]
defines a norm on $\ell_{\mathcal{A},\mathbf{p}}$.

By $h_{\mathcal{A},\mathbf{p}}$ we denote the norm closure of
the linear space generated by the canonical basis $e_n$
($e_n(k) = \delta_{k,n}$, i.e.\ $0$ if $k\neq n$ and $1$ if
$k=n$). We shall show that $h_{\mathcal{A},\mathbf{p}}$
has a polyhedral renorming, by applying Corollary
\ref{uncbasissplitting}.
\end{example}

\begin{proof}
First we mention that for $\|x\|\leq 1$, 
\[
\Phi(x)\;=\;\sup\dset{\sum_{n=0}^{\infty} \sum_{k\in B_n} |x_k|^{p(B_n)}}{B_n\subset A_n \text{ and } B_n \text{ are pairwise disjoint}},
\]
where  $p(B)=\min\set{p_n}{B \subset A_n}$.

It is easy to check that for each $x\in h_{\mathcal{A},\mathbf{p}}$,
we have $\Phi(x/\lambda)<+\infty$ for all $\lambda >0$, and $\Phi(x/\n{x})=1$
if $x\neq 0$. It is also not difficult to observe that the $e_n$
form a monotone basis of $h_{\mathcal{A},\mathbf{p}}$.

For each $m\in \N$, set $\alpha(m) = \min\set{n}{k\in A_n\text{ for some }k>m}$
and $\beta(m) = p_{\alpha(m)}$. Observe that $\alpha(m)$ and $\beta(m)$
are unbounded and increasing because the $A_n$ are finite.
If $B\subset A_n$ and $n<\alpha(m)$, then $\max\set{k}{k \in B}\leq m$,
and if $B\subset A_n$ and $n\geq\alpha(m)$, then $p(B)\geq \beta(m)$.
Given $x\in h_{\mathcal{A},\mathbf{p}}$, let $R_m(x)= x-P_m(x)$, where the $P_m$
are the standard projections introduced before Corollary \ref{uncbasissplitting}.

If $q\in (0,1)$, then for each $x \in h_{\mathcal{A},\mathbf{p}}$, there
is some $m=m(x)$ such that $\n{R_m x} < q$. Set $S_m =
\set{x\in S_{h_{\mathcal{A},\mathbf{p}}}}{m(x)=m}$.
Clearly $S_{h_{\mathcal{A},\mathbf{p}}}=\bigcup_{m=0}^{\infty}
S_m$. We are going to show that
\[
b_m \;=\; \inf\dset{\n{P_m(x)}}{x \in S_m}
\]
is strictly positive and converges to $1$.

Let $x \in S_m$. We have
\begin{align*}
1 &\;=\; \Phi(R_m(x)/ \n{R_m(x)}) \;\geq\; \Phi(R_m(x)/ q) \\
&\;\geq\; \sup\dset{\sum_{n\geq \alpha(m)}^{\infty}
\sum_{k\in B_n, k>m} \frac{|x(k)|^{p(B_n)}}{q^{\beta(m)}}}
{B_n\subset A_n,\, B_n \text{ pairwise disjoint}} \\
&\;=\; \frac{\Phi(R_m(x))}{q^{\beta(m)}}.
\end{align*}
Thus, $\Phi(R_m(x)) \leq q^{\beta(m)}$, and
\[
1 \;=\; \Phi(x) \;\leq\; \Phi(P_m(x))+\Phi(R_m(x)) \;\leq\;
\Phi(P_m(x)) + q^{\beta(m)},
\]
which yields $\Phi(P_m(x)) \geq 1-q^{\beta(m)}$.

On the other hand, since $\Phi$ is convex, $\Phi(0)=0$ and
$\n{P_m(x)} \leq \n{x} = 1$,
\[
\Phi(P_m(x)) \;\leq\; \n{P_m(x)} \Phi(P_m(x) /\n{P_m(x)}) \;=\; \n{P_m(x)},
\]
and we have $\|P_m(x)\| \geq 1- q^{\beta(m)}$. This proves that
the $b_m$ behave as required. By Corollary \ref{uncbasissplitting},
we obtain an equivalent polyhedral norm on $h_{\mathcal{A},\mathbf{p}}$
with a countable boundary having {\st}.
\end{proof}

\begin{remark}
N.\ Dew proved that if the subspace $X$ generated by the unit vector
basis of a modular Orlicz sequence space is asymptotic $\ell_\infty$,
then $X$ is isomorphically polyhedral \cite[Theorem 4.8.2]{dew}. The
spaces considered in Dew's result correspond to those of Example
\ref{new-example}, for sequences $\mathcal{A}$ of finite, pairwise
disjoint sets of integers, i.e.\ Nakano sequence spaces (see e.g.\
\cite{PR}).
\end{remark}

We can also use Corollary \ref{uncbasissplitting} to give a more
transparent and direct proof that certain Orlicz spaces $h_M$ are
isomorphically polyhedral.

\begin{example}[The polyhedral renorming of Orlicz spaces \cite{leung}]
Let $M$ be  a non-degenerate Orlicz function (i.e.\ $M(t)>0$ for all $t>0$),
with the additional property that
\begin{equation}\label{eq:hyp-leung}
\lim_{t\to 0} \frac{M(Kt)}{M(t)}\;=\;+\infty,
\end{equation}
for some constant $K>1$.  Let $h_M(\Gamma)$ be the space of all real
functions $x$ defined on $\Gamma$, satisfying $\sum_{\gamma\in \Gamma} M(x_{\gamma} /\rho)
<+\infty$ for all $\rho>0$, 
equipped with the Luxemburg norm
\[
\n{x} \;=\; \inf\set{\rho>0}{\sum_{\gamma\in \Gamma} M(x_{\gamma} /\rho) \leq 1}.
\]
Clearly the canonical unit vector basis $\{e_{\gamma}\}_{\gamma\in
\Gamma}$ of functions $e_{\gamma}(\beta)=\delta_{\gamma,\beta}$,
whose sole non-zero value is $1$ at $\beta=\gamma$, is
unconditionally monotone.

In \cite[Theorem 4]{leung}, D.\ H.\ Leung shows that $h_M(\N)$ is
isomorphically polyhedral. We are going to show that this space, as
well as $h_M(\Gamma)$ for every infinite set $\Gamma$, verifies the
hypothesis of Corollary \ref{uncbasissplitting}, using some
arguments that appear in Leung's proof.
\end{example}

\begin{proof}
Set $d_n=\inf\set{M(Kt)/M(t)}{0<|t| \leq M^{-1}(1/n)}$. Since $d_n>1$
and $\lim_n d_n=+\infty$, we can choose a sequence $b_n > 0$, such that
$b_n <(d_n-1)/d_n$ and $\lim_n b_n=1$.

Given $x \in h_M(\Gamma)$, we consider the support of $x$ as the
countable set
\[
\supp(x) \;=\; \set{\gamma\in \Gamma}{x(\gamma)\neq 0},
\]
and we can take an injective map $\pi:\N\to\Gamma$ such that
$\supp(x)\subset \pi(\N)$ and $\pi$ rearranges the support of $x$ in
a non-increasing way, i.e.\
$|x_{\pi(1)}|\geq|x_{\pi(2)}|\geq\dots\geq|x_{\pi(k)}|\geq\dots$.
Set $A_n(x)=\set{\pi(k)}{k\leq n}\subset\Gamma$.

We consider $S_n=\set{x\in S_X}{\n{P_{A_n(x)}(x)} \geq b_n}$. Evidently,
\[
1 \;\geq\; \inf\set{\sup\set{\n{P_\sigma}}{\sigma\subset\Gamma,\, \card(\sigma) = n}}{x\in S_n} \;\geq\; b_n \;\to\; 1.
\]
To finish, we check that the hypotheses of Corollary
\ref{uncbasissplitting} are fulfilled, by showing that $S_X =
\bigcup_n S_n$.

Assume the contrary, then there exists $x\in S_X\setminus \bigcup_n
S_n$, so we have $\n{P_{A_n(x)}(x)} < b_n$ for every $n\in\N$, and
by the definition of the norm of $h_M(\Gamma)$ we have
\begin{equation}
\label{eq:leung1} \sum_{\gamma\in A_{n}(x)} M(x_\gamma/b_n) \;\leq\; 1.
\end{equation}
Since $\card(A_n(x)) = n$, from \eqref{eq:leung1} we deduce
$\min\set{|x_{\gamma}|}{\gamma\in A_n(x)}\leq M^{-1}(1/n)$.

Using the fact that the basis of $h_M(\Gamma)$ is unconditionally
monotone, we have $\lim_n\n{x - P_{A_n(x)}(x)}=0$, so we can find $n
\in \N$ satisfying $\n{x - P_{A_n(x)}(x)}<K^{-1}$, meaning
\begin{equation}
\label{eq:leung2} \sum_{\gamma\not\in A_{n}(x)} M(Kx_\gamma) \;\leq\; 1.
\end{equation}

From the definition of $A_n(x)$, it follows that $|x_{\gamma}| \leq
\min\set{|x_{\beta}|}{\beta \in A_n(x)}\leq M^{-1}(1/n)$ for each
$\gamma \not\in A_n(x)$, and by the definition of $d_n$ we obtain
\begin{equation}
\label{eq:leung3} M(K x_\gamma) \;\geq\; d_n M(x_\gamma) \qquad\text{for all } \gamma\not\in A_{n}(x).
\end{equation}

Lastly, as $M$ is an Orlicz function and $b_n <1$, we have
$M(x_\gamma)\leq b_n M(x_\gamma/b_n)$. This, together with \eqref{eq:leung1},
\eqref{eq:leung2} and \eqref{eq:leung3}, yields the contradiction
\begin{align*}
1 \;=\; \sum_{\gamma\in \Gamma}M(x_\gamma)
&\;\leq\; b_n \sum_{\gamma\in A_n(x)} M(x_\gamma/b_n) + d_n^{-1}\sum_{\gamma\not\in A_n(x)} M(Kx_\gamma)\\
&\;\leq\; b_n + d_n^{-1} \;<\; 1. \qedhere
\end{align*}
\end{proof}

Let us remark that we do not know how to apply our method
to the space $h_M$, where $M$ is the Orlicz function constructed
by H\'ajek and Johanis \cite{HJ}.

Finally, we consider $C(K)$ spaces, where $K$ is compact and
in general non-metrizable.

\begin{corollary}[{\cite[Theorem 11]{FPST}}]
Let $K$ be a $\sigma$-discrete compact space. Then given $\eps>0$,
$C(K)$ admits a (polyhedral) norm $\tndot$ on $X$ that $\eps$-approximates the
original norm and admits a boundary having {\st}.
\end{corollary}

\begin{proof}
If $K = \bigcup_{n=0}^\infty D_n$, where each $D_n$ is discrete, then
$E_n = \set{\pm \delta_t}{t \in D_n}$ is $w^*$-discrete, so
$w^*$-LRC by Example \ref{weakstarlrcex} (1). The union of the $E_n$
is a $w^*$-compact boundary of $C(K)$. Finish by applying Theorem \ref{boundary}.
\end{proof}

We provide the following partial converse to Theorem \ref{boundary} that
applies to $C(K)$ spaces.

\begin{proposition}\label{sigma-discrete-converse}
Suppose that $C(K)$ admits a boundary (with respect to the canonical sup norm) 
that is covered by $E=\bigcup_{n=0}^\infty E_n$, where each $E_n$ is
$w^*$-LRC. Moreover, suppose that $E$ is $w^*$-$\Ksig$. Then
$K$ is $\sigma$-discrete.
\end{proposition}

\begin{proof}
Let $\delta_t \in C(K)^*$ denote the evaluation functional corresponding
to $t \in K$. Throughout this proof, we shall identify all subsets of $K$
with their canonical images inside $C(K)^*$. Observe that $E_n \cap K$
is a discrete subset of $K$. Indeed, if $\delta_t \in E_n \cap K$ then there
exists a $w^*$-open subset $U \subset C(K)^*$ such that $\delta_t \in U$
and $E_n \cap U$ is totally bounded. Bearing in mind that $\n{\delta_r-\delta_s}_1 = 2$
whenever $r,\,s \in K$ are distinct (where $\ndot_1$ is the usual norm on $C(K)^*$),
it follows that $E_n \cap K \cap U$ must be finite. We complete the proof
by showing that $K \subset E$.

Any boundary of any Banach space $Y$ must contain the $w^*$-exposed
points of $B_{Y^*}$. Denote by $G$ the set of all $t \in K$ for
which $\{t\}$ is a $G_\delta$ subset of $K$. For each $t \in G$, there exists
a function $f \in S_{C(K)}$ such that $f(t)=1 > |f(s)|$ for all $s \in K\setminus\{t\}$.
Therefore $\delta_t$ is a $w^*$-exposed point of $B_{C(K)^*}$
and $G \subset E \cap K$.

We claim that this forces $E \cap K = K$. Suppose otherwise. According to
Theorem \ref{boundary}, we know that $X$ admits an equivalent
polyhedral norm and is thus an Asplund space \cite{fonf2}, so $K$ is scattered.
As $E$ is $w^*$-$\Ksig$, the non-empty set $K\setminus E$
must be a $G_\delta$. Since $K$ is scattered, $K\setminus E$ contains a
relatively isolated point $t$, and as $K\setminus E$ is a $G_\delta$, this
implies $t \in G$, which is a contradiction.
\end{proof}

It is important to note that the boundary of $C(K)$ in Proposition
\ref{sigma-discrete-converse} must be taken with respect to the usual norm.

\begin{example}
Let $\Gamma$ be a set and let $K$ be a family of subsets of $\Gamma$
that is (a) compact in the topology of pointwise convergence, and
(b) hereditary, in the sense that $A \in K$ whenever $A \subseteq B$
and $B \in K$. Define $\ell_K$ to be the set of all real functions
$x$ on $\Gamma$, such that $\n{x}_K = \sup_{A \in K}
\left(\sum_{\gamma \in A} |x_\gamma|\right)$ is finite. Set $h_K =
\cl{\s}^{\ndot_K}(e_\gamma)_{\gamma \in \Gamma}$, where $e_\gamma$
is the standard unit vector. Such spaces have been studied in e.g.\
\cite{am:93} and \cite{st:09}. It is easy to check that
$\{e_\gamma\}_{\gamma\in \Gamma}$ forms a normalised 1-unconditional
basis of $h_K$. For instance, if $K$ is the set of all subsets of
$\Gamma$ of cardinality at most $1$, then $\ell_K =
\ell_\infty(\Gamma)$ and $h_k = c_0(\Gamma)$.

If every element of $K$ is finite, then $h_K$ admits an equivalent
polyhedral norm. First, the map $T:h_K \to C(K)$, given by $(Tx)(A)
= \sum_{\gamma \in A} x_\gamma$, is an isomorphism (because $K$ is
hereditary). Second,  $K$ is $\sigma$-discrete, because each set
$\set{A \in K}{\card(A)=n}$ is discrete. If $K$ contains an infinite
element then $\ell_1$ embeds into $h_K$, so it does not admit an
equivalent polyhedral norm.
\end{example}

\end{document}